\documentclass[12pt,a4paper]{article}
\usepackage{fullpage}
\usepackage[utf8]{inputenc}
\usepackage{amsfonts}
\usepackage{amsmath}
\usepackage{amssymb}
\usepackage{amsthm}
\usepackage{hyperref}
\usepackage[usenames]{color}
\usepackage{tikz}
\usepackage{tkz-graph}
\usetikzlibrary{calc,intersections,through,backgrounds,arrows,graphs,graphs.standard}
\usepackage{caption}

\usepackage[refpage]{nomencl}

\newtheorem{theorem}{Theorem}[section]
\newtheorem{lemma}[theorem]{Lemma}

\newtheorem{proposition}[theorem]{Proposition}
\newtheorem{corollary}[theorem]{Corollary}

\theoremstyle{definition}
\newtheorem{definition}[theorem]{Definition}

\newtheorem{remark}[theorem]{Remark}
\newtheorem{notation}[theorem]{Notation}

\newcommand{\II}{\mathbb I}

\newcommand{\NN}{\mathbb N}

\newcommand{\PP}{\mathbb P}

\renewcommand{\a}{\mathfrak a}
\renewcommand{\c}{\mathfrak c}
\renewcommand{\d}{\mathfrak d}

\renewcommand{\o}{\mathfrak o}

\renewcommand{\it}{\mathfrak{it}}
\newcommand{\ip}{\mathfrak{ip}}
\newcommand{\cg}{\mathfrak{cg}}
\newcommand{\g}{\mathfrak{g}}

\newcommand{\il}{\mathfrak{il}}
\newcommand{\im}{\mathfrak{im}}
\newcommand{\imp}{\mathfrak{imp}}
\newcommand{\ilm}{\mathfrak{ilm}}
\newcommand{\ogem}{\mathfrak{ogem}}

\newcommand{\A}{\mathcal A}
\newcommand{\C}{\mathcal C}
\newcommand{\D}{\mathcal D}
\newcommand{\U}{\mathcal U}

  \newcommand{\G}{\mathcal G}
  \newcommand{\CG}{\mathcal{CG}}
  \newcommand{\T}{\mathcal T}
  \newcommand{\IT}{\mathcal{IT}}

  \renewcommand{\O}{\mathcal O}
  \newcommand{\CO}{\mathcal CO}
  \renewcommand{\L}{\mathcal L}
  \newcommand{\IL}{\mathcal{IL}}
  
  \newcommand{\IP}{\mathcal{IP}}
  \newcommand{\LM}{\mathcal {LM}}
  \newcommand{\ILM}{\mathcal {ILM}}

  \newcommand{\CN}{\mathcal {CN}}
  \newcommand{\CM}{\mathcal {CM}}
  \newcommand{\CCM}{\mathcal {CCM}}
  \newcommand{\IM}{\mathcal {IM}}
  \newcommand{\GEM}{\mathcal {GEM}}
  \newcommand{\OGEM}{\mathcal {OGEM}}
  \newcommand{\COGEM}{\mathcal {COGEM}}
  \renewcommand{\S}{\mathcal {S}}

  \newcommand{\DIM}{d}
  \newcommand{\DDIM}{D}

\newcommand{\SET}{\mathrm{\sc SET}}
\newcommand{\SEQ}{\mbox{\sc SEQ}}
\newcommand{\DAG}{\mbox{\sc DAG}}

\author{Thierry Monteil\footnote{
    LIPN, CNRS (UMR 7030), Universit\'e Paris 13, F-93430 Villetaneuse, France.\newline
    Email: \texttt{thierry.monteil@lipn.univ-paris13.fr}}
    \and
    Khaydar Nurligareev\footnote{
    LIPN, CNRS (UMR 7030), Universit\'e Paris 13, F-93430 Villetaneuse, France.\newline
    LIB, Université de Bourgogne, F-21078, Dijon, France.\newline
    Email: \texttt{khaydar.nurligareev@u-bourgogne.fr}}
  }
\date{}

\title{Asymptotic probability of irreducibles I: connectedness}

\begin{document}

\maketitle

\begin{abstract}
  We study the structure of the asymptotic expansion of the probability that a
  combinatorial object is connected.
  We show that the coefficients appearing in those asymptotics are integers and
  can be interpreted as the counting sequences of derivative combinatorial
  classes.
  This result applies to rapidly growing combinatorial structures, which we
  call gargantuan, that also admit a sequence decomposition.
  We apply it to several models of graphs, of surfaces (square-tiled surfaces,
  combinatorial maps), and to geometric models of higher dimension
  (constellations, graph encoded manifolds).
  The corresponding derivative combinatorial classes are irreducible
  (multi)tournaments, indecomposable (multi)permutations and indecomposable
  perfect (multi)matchings.
  In particular, we complete the asymptotic expansion in the case of random
  labeled simple graph for which the first terms were given by Gilbert and by
  Wright.
\end{abstract}

\section{Introduction}

This paper is the first in a series that aim to provide precise bounds on the
probability that a combinatorial object is irreducible, for a notion of
irreducibility relevant to the considered combinatorial object.
This paper focuses on connectedness.

Many combinatorial structures
with a topological flavour
admit a decomposition into connected components.
We are interested in the probability that an object of such a class is
connected, when its size goes to infinity.
More precisely, for a given size $n\geq 0$, we endow the subset of objects of
size $n$ in the class with the uniform probability, and then study the
behavior of the probability that an object is connected, when $n$ goes to
infinity.

For example, the probability that a labeled graph with $n$ vertices is
connected tends to $1$ as $n$ goes to infinity~\cite{Gilbert1959}.
The probability that a permutation of size $n$ consists of a single cycle is
equal to $1/n$, which goes to $0$ as $n$ goes to infinity (the expression of a
permutation into the commutative product of cycles with disjoint support can be
understood as a decomposition of the underlying graph into its connected
components).
In between, the probability that a labeled forest of size $n$ is connected goes
to $1/\sqrt{e}$ as $n$ goes to infinity~\cite{Renyi1959}.

The question of how to distinguish between these three cases was raised by
Wright~\cite{Wright1967}.
\newline

The main tool to deal with that question is the use of generating functions. To
the counting sequence $(\a_n)$ of a labeled combinatorial class $\A$, one can
associate the \emph{generating function}
\begin{equation}\label{formula: generating function}
  A(z) = \sum\limits_{n\geq 0}\a_n\frac{z^n}{n!}.
\end{equation}

The fact that any element of $\A$ can be uniquely decomposed into the disjoint
union of its connected components translates into the following
\emph{exponential formula} (see e.g. \cite{Stanley1999}, chapter 5): 
%
\begin{equation}\label{formula: exp}
  A(z) = \exp\left(C(z)\right),
\end{equation}
where $C(z) = \sum\limits_{n=0}^{\infty}\c_n\dfrac{z^n}{n!}$ is the generating
function associated with the class $\C$ of elements of $\A$ that are connected.
Equivalently,
\begin{equation}\label{formula: log}
  C(z) = \log\left(A(z)\right),
\end{equation}
so that the number $\c_n$ of connected objects of size $n$ only depends on the
numbers $\a_k$ of all objects of sizes $k\leq n$. Hence, case-by-case
topological considerations can be avoided in estimating the probability
$\c_n/\a_n$.
\newline

 Wright showed~\cite{Wright1967,Wright1968} the following equivalence:
 \[
  \lim\limits_{n\to\infty}(\c_n/\a_n) = 1
  \quad\Leftrightarrow\quad
  \sum\limits_{k=1}^{n-1}\binom{n}{k}\a_k\a_{n-k} = o(\a_n)
  \quad\Leftrightarrow\quad
  \sum\limits_{k=1}^{n-1}\binom{n}{k}\c_k\c_{n-k} = o(\c_n)
 \]
 Wright's approach was developed further \cite{Compton1987, Cameron1997,
 BenderCameronOdlyzkoRichmond1999, Bell2000}
 and culminated in the paper of Bell, Bender, Cameron and
 Richmond~\cite{BellBenderCameronRichmond2000} who proved the following
 characterization.
 Assuming that the limit $\rho = \lim\limits_{n\to\infty}(\c_n/\a_n)$ exists,
 its value is in the following correspondence with the radius of convergence
 $R$ of the generating function $C(z)$:
\begin{eqnarray*}
  \rho=1 & \Leftrightarrow & R=0 \\
  \rho=0 & \Leftrightarrow & R>0 \mbox{\ \ and $C(R)$ diverges} \\
  0<\rho<1 & \Leftrightarrow & R>0 \mbox{\ \ and $C(R)$ converges} \\ 
\end{eqnarray*}
Moreover, they studied the case when the limit does not exist, and described
the behavior of $\lim\sup(\c_n/\a_n)$ and $\lim\inf(\c_n/\a_n)$ depending on
the parameter $R$.
\newline

The present paper focuses on the first case: the sequences $(\a_n)$ and
$(\c_n)$ grow rapidly and the generating functions $A(z)$ and~$C(z)$ have a
radius of convergence $R=0$.
Our interest is not restricted to the limiting probability. One of our aims is to establish the whole asymptotic expansion for the probability that a random object is connected. In other words, we wish to have an expression of the type
 \begin{equation}\label{formula: whole asymptotic expansion}
  \frac{\c_n}{\a_n} =
  \rho\Big(
   1 + f_1(n) + f_2(n) + \ldots + f_r(n) + o\big(f_r(n)\big)
   \Big),
 \end{equation}
 where $f_{k+1}(n)=o\big(f_k(n)\big)$ for every $k\in\NN$, as $n\to\infty$.
\newline

To fix ideas, let us first consider a concrete illustrative example, by looking
at the probability $p_n$ that a labeled simple graph of size $n$ is connected.
In 1959, Gilbert~\cite{Gilbert1959} provided the first nontrivial term of the
asymptotic expansion:
%
$$ p_n = 1 - \frac{2n}{2^{n}} + O\left(\frac{n^2}{2^{3n/2}}\right).$$
Eleven years later, Wright~\cite{Wright1970apr} provided the first three
nontrivial terms:
$$ p_n = 1 - \binom{n}{1}\frac{1}{2^{n-1}} - 2\binom{n}{3}\frac{1}{2^{3n-6}} -
24\binom{n}{4}\frac{1}{2^{4n-10}} + O\left(\frac{n^5}{2^{5n}}\right).$$
While Wright's method allows to compute more terms recursively, it does not
provide a way to grasp the structure of the whole asymptotic expansion, if only
because there is no interpretation to the coefficients $1,2,24,\dots$.
As we shall see, as a consequence of a general theorem that we are about to
announce (Theorem \ref{theorem:SET-asymptotics}), our method allows to provide
the following asymptotic for connected labeled graphs:

{
\renewcommand{\thetheorem}{\ref{proposition:graph}}
\begin{corollary}
  The asymptotic probability that a random labeled simple graph $g$ with $n$~vertices is connected satisfies
  \begin{equation*}
  \PP\big(g\mbox{ is connected}\big) \approx 1 - \sum\limits_{k\ge1} \it_k \cdot \binom{n}{k} \cdot \frac{2^{k(k+1)/2}}{2^{kn}},
  \end{equation*}
  where $\it_k$ denotes the number of irreducible tournaments of size $k$.
\end{corollary}
\addtocounter{theorem}{-1}
}

In this example, it might be surprising to see the counting sequence of some
directed objects (irreducible tournaments) appearing as coefficients in the
asymptotic of some non-directed objects (connected simple graphs).
Irreducible tournaments come from the decomposition of a tournament into a
sequence of irreducible ones.
While the class of tournaments and the class of graphs are not isomorphic (e.g.
when we swap the labels, the two tournaments of size $2$ are exchanged, while
the two graphs of size two are fixed), they share the same counting sequence,
hence they can be identified from an enumerative combinatorics perspective.
The number of irreducible tournaments plays the role of the derivative
sequence.
We could summarize the double decomposition as follows:

\begin{center}
\begin{tikzpicture}[scale=0.85]
  \small
  \draw (0.6,0) node {$\displaystyle G(z) = \sum_{n=0}^{\infty} 2^{\binom{n}{2}} \cfrac{z^n}{n!} = T(z)$};
  \draw[->] (0,0.9) -- (0,2.5);
  \draw[left, gray] (0,1.7) node {counts};
  \draw (0,3) node {graphs};
  \draw[->] (0,-0.9) -- (0,-2.5);
  \draw[left, gray] (0,-1.7) node {counts};
  \draw (0,-3) node {tournaments};
  \draw[->] (1.5,3) -- (4.5,3);
  \draw[above, gray] (3,3) node {decomposed};
  \draw[below, gray] (3,3) node {as SET of};
  \draw[above] (6,3) node {connected};
  \draw[below] (6,3) node {graphs};
  \draw[->] (1.5,-3) -- (4.5,-3);
  \draw[above, gray] (3,-3) node {decomposed};
  \draw[below, gray] (3,-3) node {as SEQ of};
  \draw[above] (6,-3) node {irreducible};
  \draw[below] (6,-3) node {tournaments};
%
  \draw[->] (7.5,3) -- (10.5,3);
  \draw[above, gray] (9,3) node {counted by} ;
  \draw[right] (10.5,3) node {$\displaystyle CG(z) = \sum_{n=0}^{\infty} \cg_n \cfrac{z^n}{n!}=\log(G(z))$};
%
  \draw[->] (7.5,-3) -- (10.5,-3);
  \draw[above, gray] (9,-3) node {counted by} ;
  \draw[right] (10.5,-3) node {$\displaystyle IT(z) = \sum_{n=0}^{\infty} \it_n \cfrac{z^n}{n!}=1-\cfrac{1}{T(z)}$} ;
  \draw[->] (15.5,2) -- (15.5,-2);
  \draw[left, red] (15.5,0) node {derivative} ;
\end{tikzpicture}
\end{center}

%

This asymptotic expansion for the probability that a simple graph is connected
comes from a more general result. 
Indeed, as we shall see in Section~\ref{section:double decomposition}, the
existence of a double SET/SEQ decomposition is not artificial and arises
naturally from a simple computation.

When the sequence $(\a_n)$ grows fast enough (see Definition~\ref{def:
gargantuan sequence} of a gargantuan class), the whole asymptotic expansion
(\ref{formula: whole asymptotic expansion}) for the probability that a random
object is connected can be made explicit, and depends on a series of integer
coefficients that can be interpreted as the counting sequence $(\d_n)$ of a
\emph{derivative} class $\D$ that depends on the decomposition of the elements
of $\A$ into a sequence of elements of a class $\D$ (see
Definition~\ref{def:SET/SEQ} of a derivative class and of a double
decomposition).

Our main result is:

{
\renewcommand{\thetheorem}{\ref{theorem:SET-asymptotics}}
\begin{theorem}[$\SET$ asymptotics]
  Let $\A$ be a gargantuan labeled combinatorial class that admits a double decomposition $\A = \SET(\C) = \SEQ(\D)$ for some labeled combinatorial
  classes $\C$ and~$\D$.
  Suppose that $a\in\A$ is a random object of size~$n$.
  Then
  \begin{equation*}
    \PP(a\mbox{ is }\SET\mbox{-irreducible}) \approx 1 - \sum\limits_{k\ge
    1}\d_k\cdot\binom{n}{k}\cdot \frac{\a_{n-k}}{\a_n}.
  \end{equation*}
\end{theorem}
\addtocounter{theorem}{-1}
}


It turns out that combinatorial objects others than graphs admit such a double
SET/SEQ decomposition, so that Theorem~\ref{theorem:SET-asymptotics} can be
applied in several different contexts.\\

First, beyond simple graphs, Theorem~\ref{theorem:SET-asymptotics} can be
applied to provide the whole asymptotic expansion for the probability of the
(weak) connectivity for several models of graphs: multigraphs
(Proposition~\ref{proposition:multigraph}), oriented graph
(Corollary~\ref{proposition:oriented:graph}),
digraphs(Corollary~\ref{proposition:digraph}). The corresponding derivative
classes involve irreducible multitournaments, that generalize irreducible
tournaments.

In a more geometrical setting, we will apply
Theorem~\ref{theorem:SET-asymptotics} to provide the whole asymptotic expansion
for the probability of the connectedness for some discrete surface models:
square-tiled surfaces (Proposition~\ref{proposition:origami}), whose derivative
class involve indecomposable permutations (or irreducible pairs of linear
orders) and combinatorial maps (Proposition~\ref{proposition:SET-asymptotics
for combinatorial map surfaces}), whose derivative class involve indecomposable
perfect matchings (or irreducible linear matchings).


Looking at higher dimensional models, we will provide the asymptotic expansion
for the probability of the connectedness of GEMs -- Graph Encoded Manifolds
(Proposition~\ref{proposition:OGEM}) and constellations
(Proposition~\ref{proposition:constellation}), whose corresponding derivative
classes are indecomposable multipermutations.

Also, Theorem~\ref{theorem:SET-asymptotics} can be applied even if no explicit
SEQ decomposition is known, see Proposition~\ref{proposition:noSEQ} for some
examples where we provide estimates for the probability that a triangulated
surface, a quadrangulated surface, a quadratic square-tiled surface, a
(non-oriented) 3-dimensional GEM, is connected.

\section{Tools}

\subsection{Asymptotic expansion}

\label{subsection: Asymptotic expansion}
\begin{notation}\label{notation: approx} 
 For a sequence $(a_n)$ of numbers and an integer $m$, we write\nomenclature[W$\Z$]{$\approx$}{asymptotic expansion}
  \begin{equation}\label{formula: asymptotic expansion}
   a_n \approx \sum\limits_{k\ge m}f_k(n)
  \end{equation}
  if
  \begin{equation*}
    \forall r\ge m, \ \ a_n = \sum\limits_{k=m}^{r}f_k(n) + O\big(f_{r+1}(n)\big)
  \end{equation*}
  and
  \begin{equation*}
      \forall k\geq m, \ \ f_{k+1}(n) = o\big(f_k(n)\big).
  \end{equation*}
The expression \eqref{formula: asymptotic expansion} is called an \emph{asymptotic expansion} of the sequence $(a_n)$.
\end{notation}

\subsection{Combinatorial classes and decompositions}

We use the standard notion of \emph{labeled combinatorial structure (or
class)}, as used in the textbooks \cite{Stanley1999, FlajoletSedgewick2009,
Bona2015}.
Briefly, and to fix the notations, a combinatorial structure $\A$ is a
collection of objects of finite size such that for each integer $n\geq 0$, the
number $\a_n$ of objects of size $n$ is finite.
The class $\A$ is \emph{labeled} if the objects of size $n$ are defined on a
set of labels of cardinality $n$, whose elements are distinguishable (hence,
the elements of $\A$ have only trivial automorphisms). It is common to fix the
ground set to $[n]=\{1,\dots,n\}$, but we should keep in mind that the nature
of the ground set is irrelevant. The notion of combinatorial
species~\cite{Joyal1981, BergeronLabelleLeroux1998} makes this point precise.
In this paper, every combinatorial structure is assumed labeled.
To such a structure $\A$, we associate its generating function
$$A(z) = \sum\limits_{n\geq 0}\a_n\frac{z^n}{n!}.$$

Combinatorial classes can be combined to create new classes: for example, we
can construct the disjoint union or the labeled product of finitely many
combinatorial classes. The corresponding generating functions are the sum and
the product of the original ones, respectively.
A combinatorial class can be combined with itself in various ways.
We can form the \emph{sequence} of a single class $\A$ by considering the
labeled products of $\A$ by itself, involving arbitrary many terms.
We denote $\SEQ(\A)$ the corresponding class. Its generating function, denoted
by $\SEQ(\A)(z)$, satisfies:
$$\SEQ(\A)(z)=\sum\limits_{k \geq 0}A(z)^k=\frac{1}{1-A(z)}.$$
The \emph{set} of a single class $\A$ is obtained from the sequence $\SEQ(\A)$
by forgetting the order of the components.
We denote $\SET(\A)$ the corresponding class. Its generating function, denoted
by $\SET(\A)(z)$, satisfies:
$$\SET(\A)(z)=\sum\limits_{k \geq 0}\frac{A(z)^k}{k!}=\exp(A(z)).$$

The $\SET$ and $\SEQ$ constructions are only possible when $A(0)=\a_0=0$,
in which case we have $\SET(\A)(0)=1$.

%
%

%
To a combinatorial class $\A$ and an integer $n$ such that $\a_n>0$, we endow
the nonempty finite set $\A_n$ of objects of size $n$ with the uniform
probability $\PP_n$: each object of size $n$ has probability~$1/\a_n$.
If $Q$ is some property about objects of $\A$, we denote by
$$\PP(a \mbox{ satisfies }Q)$$
the sequence of probabilities
$$\big(\PP_n\{a \in \A_n\mid a \mbox{ satisfies }Q\}\big)_{\a_n>0}.$$
\newline

When the objects of a class $\A$ have a topological flavour so that
connectedness makes sense, we can consider the subclass $\C$ of connected
objects from $\A$. In such a case, when the connected components of objects
from $\A$ belong to $\C$ and when the knowledge of the connected components of
an object from $\A$ is sufficient to reconstruct it, we have $\A=\SET(\C)$.
In this case, we have $C(z)=\log(A(z))$, so that the probability 
$$\PP_n(a \mbox{ is connected}) = \c_n/\a_n$$
only depends on the counting sequence $(\a_n)$.
Hence, we will identify labeled combinatorial classes if they share the same
counting sequence
(such classes are called ``combinatorially isomorphic'' in
\cite{FlajoletSedgewick2009}). This identification is adapted to our objective
and offers a lot of flexibility.

Note that the fact that objects of a combinatorial structure admit a
topological model is not sufficient to state that it is the $\SET$ of its
connected objects.
For example, chord diagrams~\cite{Stein1978} and meandric
systems~\cite{LandoZvonkin1992} have a topological representation, but the
knowledge of their connected components is not enough to reconstruct the
original objects, since connected objects can be intertwined in several ways.

\subsection{Gargantuan classes and Bender theorem}
\label{section: Bender}

\begin{definition}\label{def: gargantuan sequence}
A sequence $(a_n)$ is \emph{gargantuan} if, for any integer $r$, as $n\to\infty$, the following two conditions hold:
 \[
  \mbox{ (i) }\quad
  \frac{a_{n-1}}{a_n} \to 0;
  \qquad\qquad
  \mbox{ (ii) }\quad
  \sum\limits_{k=r}^{n-r}|a_ka_{n-k}| = O\big(a_{n-r}\big).
 \]

 A labeled combinatorial class $\A$ is \emph{gargantuan} if its counting
  sequence $(\a_n)$ is such that $(\a_n/n!)$ is gargantuan.
\end{definition}

\begin{theorem}[Bender~\cite{Bender1975}, in the simplified form of~\cite{Odlyzko1995}]
\label{theorem:Bender}
 Consider the formal power series
 \[
  U(z) = \sum\limits_{n=1}^{\infty}u_nz^n
 \]
 and a function $F(x)$ analytic in a neighborhood of the origin.
 Define
 \[
  V(z) = \sum\limits_{n=0}^{\infty}v_nz^n = F\big(U(z)\big)
 \qquad
 \mbox{and}
 \qquad
  W(z) = \sum\limits_{n=0}^{\infty}w_n z^n =
  \left.\frac{\partial}{\partial x} F(x)\right|_{x=U(z)}.
 \]
 Assume that $u_n\ne 0$ for all sufficiently large $n$, and that the sequence $(u_n)$ is gargantuan.
 Then
 \[
  v_n \approx \sum\limits_{k\ge0}w_ku_{n-k}
 \]
 and the sequence $(v_n)$ is gargantuan.
\end{theorem}

Most gargantuan sequences that we meet in this paper show the same typical behavior.
In order to describe this behavior, for a given gargantuan sequence $(a_n)$ and large $n$, let us introduce a finite sequence $x_k = |a_{k}a_{n-k}|$.
Clearly, $(x_k)$ is symmetric with respect to $k=n/2$.
Taking into account the second condition of Definition~\ref{def: gargantuan sequence}, it is not surprising that $x_1>x_2$, $x_2>x_3$, etc.
However, in most cases we consider, the behavior of $(x_k)$ is even more strict and the inequality $x_k>x_{k+1}$ holds for all $k<n/2$ (Figure~\ref{figure: x_k behavior}).
We will use this observation through the following lemma.

\begin{figure}[ht!]
\begin{center}
\begin{tikzpicture}[scale=1.2]
 \draw[->] (-2.4,0) -- (3.2,0);
 \draw[->] (-2.4,0) -- (-2.4,4.2) node[below left] {$x_k$};
 \draw[dotted] (-2,0) -- (-2,3.5);
 \draw[dotted] (0,0) -- (0,4.2);
 \draw[dotted] (2,0) -- (2,3.5);
 \draw (-2,0) node[below] {${\color{white}|}1{\color{white}|}$};
 \draw (0,0) node[below] {${\color{white}|}n/2{\color{white}|}$};
 \draw (2,0) node[below] {${\color{white}|}n{\color{white}|}$};
 \draw (3,0) node[below] {${\color{white}|}k{\color{white}|}$};
 \foreach \i [evaluate=\i as \xx using (-2 + 4*\i/19),
              evaluate=\xx as \yy using 3*\xx*\xx/4 + 0.5] in {0,...,19}
   {\fill (\xx,\yy) circle (1.4pt);}
\end{tikzpicture}
\end{center}
\caption{Typical behavior of the sequence $(x_k)$.}\label{figure: x_k behavior}
\end{figure}


\begin{lemma}\label{lemma: sufficient conditions for gargantuan sequence}
 If a sequence $(a_n)$ satisfies the following two conditions
 \begin{align*}\label{equation: sufficient conditions for gargantuan sequence}
  \emph{ (i)' }\quad &
  na_{n-1} = O(a_n),\,\,\mbox{ as }\, n\to\infty, \\
  \emph{ (ii)' }\quad &
  x_k = |a_ka_{n-k}| \mbox{ is decreasing for } k < n/2 \mbox{ and for all but finitely many } n,
 \end{align*}
 then $(a_n)$ is gargantuan.
\end{lemma}
\begin{proof}
 The first condition of Definition~\ref{def: gargantuan sequence} immediately follows from \mbox{(i)'}.
 To check the second condition, note that
 the value of $x_k$ is symmetric with respect to $k = n/2$ and it is increasing for $k < n/2$.
 Hence, conditions (i)' and (ii)' imply that
 \[
  \sum\limits_{k=r}^{n-r} |a_ka_{n-k}| =
  \sum\limits_{k=r}^{n-r} x_k =
  2x_r + \sum\limits_{k=r+1}^{n-r-1} x_k \le
  2x_r + (n-2r-1)x_{r+1} = O(a_{n-r}).
 \]
\end{proof}

%
\begin{lemma}\label{lemma: a_nb_n is gargantuan}
 If $(a_n)$ and $(b_n)$ are two non-negative gargantuan sequences, then the sequence $(a_nb_n)$ is non-negative gargantuan as well.
\end{lemma}
\begin{proof}
 Let us check the conditions of Definition~\ref{def: gargantuan sequence}. The first condition trivially holds. The second condition reads
 \[
  \sum\limits_{k=r}^{n-r}(a_kb_k)(a_{n-k}b_{n-r}) \le
  \left(\sum\limits_{k=r}^{n-r}a_ka_{n-k}\right) 
  \left(\sum\limits_{k=r}^{n-r}b_kb_{n-k}\right) =
  O(a_{n-r})O(b_{n-r}) = O(a_{n-r}b_{n-r}).
 \]
 The condition $a_nb_n \ge 0$ is trivial.
\end{proof}

\section{Double SET/SEQ decomposition}
\label{section:double decomposition}
The goal of this section is to see how the existence of a $\SEQ$-decomposition
is a natural hypothesis for the computation of an asymptotic expansion for the
probability that an object is $\SET$-irreducible.
%
%
To this end, we will estimate this probability by a direct computation,
under the hypothesis that the counting sequence of the studied class grows
rapidly.
\newline

Let us start with two labeled combinatorial classes $\A$ and $\C$, such that
$\A = \SET(\C)$.
Each object $a\in\A$ of size $n$ consists of several connected components that
are considered as objects from $\C$.
The sizes of these components determine a partition $\lambda \vdash n$, which
can be expressed as an $n$-tuple $p$ of non-negative integers counting the
number of components for each size: $p\in P_n = \{(p_1,\ldots,p_n) \mid
p_1+2p_2+\ldots+np_n=n\}$.
The exponential formula~\eqref{formula: exp} can be expanded to obtain the following
expression for $\a_n$:
\begin{equation}\label{formula: a_n in terms of c_n}
 \a_n = \sum\limits_{p\in P_n}
  \frac{n!}{(1!)^{p_1}\ldots(n!)^{p_n}} \cdot
   \frac{1}{p_1!\ldots p_n!} \cdot
    \big(\c_1^{p_1}\ldots\c_n^{p_n}\big).
\end{equation}

If the sequence $(\a_n)$ grows fast enough, then the radius of convergence of
$A(z)$ is zero, hence $\a_n \sim \c_n$~\cite{BellBenderCameronRichmond2000}.
Let us look at the asymptotic expansion of the difference $(\a_n-\c_n)$.
For gargantuan sequences, this difference is decomposed into a sum of
$\c_{n-k}$, $k=1,2,3,\ldots$, taken with certain coefficients that can be
deduced from~\eqref{formula: a_n in terms of c_n}.
For instance, the first three terms of the asymptotics are
\[
 \a_n - \c_n
  = 
 \binom{n}{1} \c_1 \c_{n-1}
  +
 \binom{n}{2} (\c_2+\c_1^2) \c_{n-2}
  +
 \binom{n}{3} (\c_3+3\c_2\c_1+\c_1^3) \c_{n-3}
  +
 \ldots
\]
or, even simpler,
\[
 \a_n - \c_n
  = 
 \binom{n}{1} \a_1 \c_{n-1}
  +
 \binom{n}{2} \a_2 \c_{n-2}
  +
 \binom{n}{3} \a_3 \c_{n-3}
  +
 O(n^4\c_{n-4}).
\]
Thus, a finite number of terms of the asymptotic expansion are determined by
objects possessing a large connected component.

In practice, we suppose that $(\a_n)$ is known, while $(\c_n)$ is not.
That is the reason why we would like to have an asymptotic expression in terms of $\a_{n-k}$ rather than $\c_{n-k}$.
In order to get one, let us apply the inclusion-exclusion principle.

The main idea is to count objects with respect to large components, replacing $\c_{n-k}$ by $\a_{n-k}$ that is asymptotically the same.
For instance, if we are interested only in the leading term of $(\a_n-\c_n)$, as $n\to\infty$, then it is sufficient to replace $\c_{n-1}$ by $\a_{n-1}$:
\[
 \a_n - \c_n
  = 
 \binom{n}{1} \c_1 \a_{n-1}
  +
 O(n^2\a_{n-2}).
\]
Combinatorially, the term $\binom{n}{1} \c_1 \a_{n-1}$ counts objects of size
$n$ with at least one isolated vertex.
Any object with $k$ isolated vertices is counted $k$ times, since we mark one
of them.
However, the contribution of that object is significantly smaller, and hence,
can be ignored.

When the second term is needed, we must accurately count objects possessing
connected components of size at least $(n-2)$.
Therefore, we need to deduce objects with connected components of sizes
$(n-2)$, $1$ and $1$, which are counted twice in the first approximation: 
\[
 \a_n - \c_n
  = 
 \binom{n}{1} \c_1 \a_{n-1}
  +
 \binom{n}{2} (\c_2 - \c_1^2) \a_{n-2}
  +
 O(n^3\a_{n-3}).
\]
Again, combinatorially, objects with several components of size $1$ or $2$ are
counted several times within these two terms, but their contribution is
$O(n^3\a_{n-3})$.

The same approach works in searching for more terms of the asymptotic expansion.
Indeed, given a positive integer $n$, for any partition $\lambda=(\lambda_1,\ldots,\lambda_{l(\lambda)})$ of an integer $k\le n$,
define $B_{\lambda,n}$ to be the number of objects with marked connected components of sizes $\lambda_1,\lambda_2,\ldots,\lambda_{l(\lambda)}$.
For example, if $n=3$, $k=2$ and $\lambda=(1,1)$, then $B_{(1,1),3}$ is three times the number of objects of size $3$ with $3$ isolated vertices: $B_{(1,1),3}=3\c_1^3$.
In the general case, $B_{\lambda,n}$ satisfies
\[
 B_{\lambda,n}
  = 
 \binom{n}{k}
 \frac{k!}{(1!)^{p_1} \ldots (k!)^{p_k}}
 \frac{\c_1^{p_1} \ldots \c_k^{p_k}}
 {p_1! \ldots p_k!}
 \a_{n-k}
\]
and, for a given number $r$ of terms of the asymptotic expansion, the inclusion-exclusion principle provides
\begin{equation}\label{formula: a_n-c_n in terms of B}
 \a_n - \c_n
  = 
 \sum\limits_{\lambda \colon |\lambda| < r}
  (-1)^{l(\lambda)-1} B_{\lambda,n}
  +
 O(n^{r}\a_{n-r}).
\end{equation}

Let us factor out $\binom{n}{k}\a_{n-k}$ in the relation~\eqref{formula:
a_n-c_n in terms of B}.
The result is of the form
$$ \a_n - \c_n \approx \sum\limits_{k \ge 1} \binom{n}{k} \d_k \a_{n-k},$$
where
\begin{equation}\label{formula: definition of d_k}
  \d_k = \sum\limits_{P_k} (-1)^{(p_1+\ldots+p_k)-1}\cdot
  \frac{k!}{(1!)^{p_1}\ldots(k!)^{p_k}}\cdot
  \frac{\c_1^{p_1}\ldots\c_k^{p_k}}{p_1!\ldots p_k!}.
\end{equation}
It turns out that these coefficients might have a combinatorial meaning on
their own.
In order to make this fact clear, let us apply the inclusion-exclusion
principle, so that we obtain an exact formula for $\a_n$:
\[
 \a_n
  = 
 \sum\limits_{\lambda \colon |\lambda| \le n}
  (-1)^{l(\lambda)-1} B_{\lambda,n}
\]
or
\[
 \a_n = \sum\limits_{k=1}^{n}
 \d_k\binom{n}{k}\a_{n-k}.
\]
If $D(z)$ denotes the power series $D(z) = \sum\limits_{n\ge 1}\d_n z^n/n!$,
the last equation can be rewritten as
\[
 A(z) = 1 + A(z) D(z).
\]
Equivalently,
\[
 A(z) = \frac{1}{1 - D(z)},
\]
so that if we could find a labeled combinatorial class $\D$ satisfying
$\A=\SEQ(\D)$, then $(\d_n)$~would be the counting sequence of $\D$.
\newline

This leads us to the following definition:
\begin{definition}\label{def:SET/SEQ}
  A combinatorial class $\A$ \emph{admits a double $\SET/\SEQ$ decomposition}
  if there exist two classes $\C$ and $\D$ such that $\A=\SET(\C)=\SEQ(\D)$.
  The class $\D$ is called the \emph{derivative class} of $\A$ and the counting
  sequence $(\d_n)$ is called the \emph{derivative sequence} of the sequence
  $(\a_n)$.
\end{definition}

\section{Main results}

We refer to Definition~\ref{def: gargantuan sequence} of a gargantuan class,
and Definition~\ref{def:SET/SEQ} of a double decomposition and derivative
sequence.

\begin{theorem}[$\SET$ asymptotics]\label{theorem:SET-asymptotics}
  Let $\A$ be a gargantuan labeled combinatorial class that admits a double decomposition $\A = \SET(\C) = \SEQ(\D)$ for some labeled combinatorial
  classes $\C$ and~$\D$.
  Suppose that $a\in\A$ is a random object of size~$n$.
  Then
  \begin{equation}\label{formula: SET-asymptotics}
    \PP(a\mbox{ is }\SET\mbox{-irreducible}) \approx 1 - \sum\limits_{k\ge
    1}\d_k\cdot\binom{n}{k}\cdot \frac{\a_{n-k}}{\a_n}.
  \end{equation}
\end{theorem}

\begin{proof}
  Let us consider the labeled combinatorial class $\U=\A-\{\epsilon\}$ with
  generating function $U(z)$, and the function
  $$F(y) = \log(1+y).$$
  Since $\U$ is gargantuan and $F$ is analytic at the origin, we can apply
  Theorem~\ref{theorem:Bender}.
  \newline
  We have
  $$V(z)=\log(A(z))=C(z)$$
  and
  $$W(z)=\left.\frac{\partial}{\partial x} F(x)\right|_{x=A(z)-1} =
  \frac{1}{A(z)} = 1 - D(z).$$
  Hence,
  %
  $$v_n \approx \sum\limits_{k\geq0}w_k u_{n-k}.$$
  By substituting with the original counting sequences, we have
  $$\frac{\c_n}{n!} \approx \frac{\a_n}{n!} - \sum\limits_{k\geq 1}\frac{\d_k}{k!} \frac{\a_{n-k}}{(n-k)!}.$$
  Dividing by $\frac{\a_n}{n!}$, we obtain
  $$\frac{\c_n}{\a_n} \approx 1 - \sum\limits_{k\geq 1}\binom{n}{k}\d_k \frac{\a_{n-k}}{\a_n}.$$
\end{proof}

\begin{remark}
  Theorem \ref{theorem:SET-asymptotics} provides, for each $k$, a convergence
  result when $n$ goes to infinity. If we fix $n$ in the right-hand side of
  \eqref{formula: SET-asymptotics}, the sum will be $0$.
\end{remark}

\begin{remark}
  In the proof of Theorem~\ref{theorem:SET-asymptotics}, the relationship
  between the $\SET$ and the $\SEQ$ decompositions is witnessed by the fact
  that
  $$\frac{\partial\log(z)}{\partial z} = \frac{1}{z}.$$
  This fact also supports the name ``derivative sequence''.
\end{remark}

\subsection{Asymptotics for p-periodic sequences}\label{section: SET-asymptotics d-gargantuan}

 Counting sequences of certain combinatorial classes possess zeroes.
 For example, perfect matchings only exist when the number of vertices is even.
 Similarly, some combinatorial structures admit elements only
 when the number of vertices is a multiple of~$p$ for some integer $p>1$.

 \begin{definition}\label{def:periodic}
 Let $p\geq 1$ be an integer. A sequence $(\a_n)$ is $p$-\emph{periodic} if:
 \[
  \a_n \neq 0
  \qquad\Leftrightarrow\qquad
  \exists k\in\NN, \ \ n = pk.
 \]
 \end{definition}

 This situation is quite general for a combinatorial class $\A$ that admits a
 $\SET$ decomposition $\A = \SET(\C)$: if its counting sequence $(\a_n)$
 vanishes for infinitely many $n$,
 then the sequence $(\c_n)$ is $p$-periodic for some $p>1$~\cite{Wright1967,
 Wright1968}.
 In particular, both sequences $(\a_{pn})$ and $(\c_{pn})$ are eventually
 positive.
 
 Let us adapt Theorem~\ref{theorem:SET-asymptotics} for labeled combinatorial
 classes whose counting sequences are $p$-periodic.
 This will be useful in Sections~\ref{section:combinatorial-map}
 and~\ref{section:GEM}.

\begin{proposition}\label{theorem:SET-asymptotics, d-gargantuan}
  Let $\A$ be a labeled combinatorial class that admits a double decomposition 
  $\A = \SET(\C) = \SEQ(\D)$ for some labeled combinatorial classes $\C$ and~$\D$.
  Assume that the counting sequence $(\a_n)$ is $p$-periodic for some $p\geq
  1$, and that the sequence $(\a_{pn}/(pn)!)$ is gargantuan.
  Suppose also that $a\in\A$ is a random object of size~$pn$.
  Then
  \begin{equation}\label{formula: SET-asymptotics, d-gargantuan}
    \PP(a\mbox{ is }\SET\mbox{-irreducible}) \approx
  1 - \sum\limits_{k\ge1}\d_{pk}\cdot \binom{pn}{pk} \cdot \frac{\a_{p(n-k)}}{\a_{pn}}.
 \end{equation}
\end{proposition}

\begin{proof}
  Conditions of the proposition ensure that we can apply Theorem~\ref{theorem:Bender} to the formal power series $U(z) = A(z^{1/p}) - 1$ and the function
  $F(y) = \log(1+y)$.
  Repeating the reasoning of the proof of Theorem~\ref{theorem:SET-asymptotics} leads to~\eqref{formula: SET-asymptotics, d-gargantuan}.
\end{proof}

\section{Applications}

\subsection{Graphs}

%
The asymptotic probability for the connectivity of several variants of graphs
can be approached with our methods. The following definitions recall the
variants we are interested in, with their derivative counterparts:


\begin{itemize}
  \item For $d\geq 1$, a (labeled) \emph{$d$-multigraph} of size $n$ is a graph
  defined on the set of vertices $[n]$, such that any pair of distinct vertices
  are joined by at most $d$ indistinguishable edges.
  \item A (labeled) \emph{$d$-multitournament} of size $n$ is a directed graph
  defined on the set of vertices $[n]$, such that any ordered pair of distinct
  vertices $i\ne j$ is joined by $d$ directed edges (which consist of $l$
  indistinguishable directed edges from $i$ to $j$ and $d-l$ indistinguishable
  directed edges from $j$ to $i$ for some $0\leq l\leq d$ that varies from pair
  to pair).
  In the sports metaphor,
    \begin{itemize}
      \item a $2$-multitournament is a tournament where ties are allowed and
        provide partial win to both teams, as seen in most football leagues
        (those will be used in~\cite{MonteilNurligareevANTISEQ}),
      \item a $3$-multitournament is a tournament whithout ties which records
        clean sheets (represented as 
          \begin{tikzpicture}[scale=1, >=latex, yshift=1cm, line width=.5pt]\draw[white] (0,0) -- (1,0); \draw[->>>] (0,0.1) -- (1,0.1);\end{tikzpicture}
        ), as seen in table football (in terms of connectivity,
        losing the game by avoiding a clean sheet (represented as
          \begin{tikzpicture}[scale=1, >=latex, yshift=1cm, line width=.5pt]\draw[white] (0,0) -- (1,0); \draw[above, <->>] (0,0.1) -- (1,0.1);\end{tikzpicture}
        ) provides a partial win).
   %
   %
    \end{itemize}

\begin{minipage}[t]{\linewidth}\label{picture:multitournaments}
\begin{center}
\begin{tikzpicture}[scale=1, >=latex, line width=.5pt]
  \tikzset{VertexStyle/.append style={inner sep=2pt}}
  \begin{scope}[xshift=0cm]
    \Vertices{circle}{1,2,3,4,5}
    \draw[->] (1) -- (2);
    \draw[->] (2) -- (3);
    \draw[->] -- (3) -- (1);
    \draw[->] (4) -- (5);
    \foreach \i in {4,5}
    {
      \foreach \j in {1,2,3}
      {
        \draw[->] -- (\i) -- (\j);
      }
    }
  \end{scope}
  \tikzset{VertexStyle/.append style={inner sep=2pt}}
  \begin{scope}[xshift=3cm]
    \Vertices{circle}{1,2,3,4,5}
    \draw[<->] (1) -- (2);
    \draw[->>] (2) -- (3);
    \draw[->>] -- (3) -- (1);
    \draw[<->] (4) -- (5);
    \foreach \i in {4,5}
    {
      \foreach \j in {1,2,3}
      {
        \draw[->>] -- (\i) -- (\j);
      }
    }
  \end{scope}
  \begin{scope}[xshift=6cm]
    \Vertices{circle}{1,2,3,4,5}
    \draw[<->>] (1) -- (2);
    \draw[->>>] (2) -- (3);
    \draw[<->>] (3) -- (1);
    \draw[->>>] (4) -- (5);
    \draw[<<<-] (2) -- (4);
    \draw[<<->] (2) -- (5);
    \draw[<<<-] (3) -- (5);
    \draw[<<<-] (3) -- (4);
    \draw[->>>] (4) -- (1);
    \draw[<->>] (5) -- (1);
  \end{scope}
%
%
\end{tikzpicture}
\captionof{figure}{A $d$-multitournament for $d=1,2,3$, corresponding to tournaments,
tournaments with ties, and tournaments recording clean sheets,
respectively.
Their decompositon as a sequence of irreducible components are
$(\{4\},\{5\},\{1,2,3\})$, $(\{4,5\},\{1,2,3\})$, and $(\{4\},\{1,2,3,5\})$,
respectively.
}
\end{center}
\end{minipage}

%
  \item A \emph{simple graph} is a $1$-multigraph, a \emph{tournament} is a
  $1$-multitournament.
\end{itemize}
A $d$-multitournament is \emph{reducible} if there exists a partition of its
  vertices into two nonempty parts $A$ and $B$ such that any pair of vertices
  $(a,b) \in A \times B$ are joined by $d$ directed edges that all go from~$a$
  to~$b$.
  A $d$-multitournament is \emph{irreducible} if it is not reducible.
  We denote by $\G(d), \CG(d), \T(d),$ and $\IT(d)$ the combinatorial classes
  of $d$-multigraphs, connected $d$-multigraphs, $d$-multitournaments, and
  irreducible $d$-multitournaments, respectively.
%

\begin{proposition}\label{proposition:multigraph}
  The asymptotic probability that a random labeled $d$-multigraph $g$ with
  $n$~vertices is connected satisfies
  \begin{equation}\label{formula: SET-asymptotics for multigraphs}
    \PP\big(g\mbox{ is connected}\big) \approx 1 - \sum\limits_{k\ge1} \it_k(d)
    \cdot \binom{n}{k} \cdot \frac{(d+1)^{k(k+1)/2}}{(d+1)^{kn}},
  \end{equation}
  where $\it_k(d)$ denotes the number of irreducible $d$-multitournaments of
  size $k$.
\end{proposition}

\begin{proof}
  First, we have $\G(d) = \T(d)$ since both classes have $(d+1)^{\binom{n}{2}}$
  objects of size $n$ for any $n$.
  The decomposition $\G(d) = \SET(\CG(d))$ follows from the decomposition of a
  graph into its connected components.
  To prove that $\T(d) = \SEQ(\IT(d))$, let us consider a $d$-multitournament
  $t$. We decompose $t$ into its irreducible components (note that the
  irreducible components of $t$ are also its strongly connected
  components~\cite{Rado1943, Roy1958}). If $A$ and $B$ are distinct components,
  we write $A<B$ if for every pair of vertices $a\in A$ and $b\in B$, every
  edge between $a$ and $b$ is directed from $a$ to $b$. We check that $<$ is a
  linear order on the set of irreducible components of $t$, so that we can
  enumerate them as $t_1<t_2<\dots<t_l$ and construct the sequence of
  irreducible $d$-multitournaments $(t_1, ..., t_l)$ out of $t$.
  Conversely, if $(t_1, ..., t_l)$ is a (labeled) sequence of irreducible
  $d$-multitournaments, we can construct a $d$-multitournament $t$ as the
  disjoint union of the $t_i$ and then adding $d$ directed edges from any
  vertex of $t_i$ to any vertex of $t_j$ if $i<j$.
  \newline
  \newline
 Let us show that the sequence
 \[
   a_n = \frac{(d+1)^{\binom{n}{2}}}{n!}
 \]
 is gargantuan by applying Lemma~\ref{lemma: sufficient conditions for gargantuan sequence}. Condition (i)' follows from
 \[
  \frac{a_{n-1}}{a_n} =
  \frac{(d+1)^{\binom{n-1}{2}}}{(n-1)!} \cdot \frac{n!}{(d+1)^{\binom{n}{2}}} = \frac{n}{(d+1)^n} \to 0.
 \]
Regarding condition (ii)', let us check that the sequence $x_k = a_k a_{n-k}$
  is decreasing for $k < n/2$.
  To this end, let us consider
 \[
  \frac{x_{k+1}}{x_k} =
  \frac{a_{k+1}a_{n-k-1}}{a_ka_{n-k}} =
  \frac{(n-k)}{(k+1)} \cdot \frac{(d+1)^{k+1}}{(d+1)^{n-k}}.
 \]
 Since the function
 \[
   f(x) = \frac{(d+1)^x}{x}
 \]
 is increasing for large $x$, we have the following equivalences for $n$ large enough:
 \[
  \frac{x_{k+1}}{x_k} \le 1
  \quad\Leftrightarrow\quad
  \frac{(d+1)^{k+1}}{(k+1)} \le \frac{(d+1)^{n-k}}{(n-k)}
  \quad\Leftrightarrow\quad
  {k+1} \le {n-k}.
 \]
 Hence, $(x_k)$ is decreasing for $k < n/2$.
 According to Lemma~\ref{lemma: sufficient conditions for gargantuan sequence}, $(a_n)$ is gargantuan.
 \newline
 \newline
Therefore, we can apply Theorem~\ref{theorem:SET-asymptotics}:
\begin{eqnarray*}
  \PP\big(g\mbox{ is }\SET\mbox{-irreducible}\big) & \approx &
  1 - \sum\limits_{k\ge1} \it_k(d) \cdot \binom{n}{k} \cdot \frac{\g_{n-k}(d)}{\g_n(d)} \\
  & \approx &
  1 - \sum\limits_{k\ge1} \it_k(d) \cdot \binom{n}{k} \cdot \frac{(d+1)^{k(k+1)/2}}{(d+1)^{kn}}. \\
\end{eqnarray*}
\end{proof}

\begin{corollary}\label{proposition:graph}
  The asymptotic probability that a random labeled simple graph $g$ with $n$~vertices is connected satisfies
  \begin{equation}\label{formula: SET-asymptotics for graphs}
  \PP\big(g\mbox{ is connected}\big) \approx 1 - \sum\limits_{k\ge1} \it_k \cdot \binom{n}{k} \cdot \frac{2^{k(k+1)/2}}{2^{kn}},
  \end{equation}
  where $\it_k$ denotes the number of irreducible labeled tournaments of size $k$.
\end{corollary}

\begin{proof}
Apply Proposition~\ref{proposition:multigraph} with $d=1$.
\end{proof}

\begin{remark}
The counting sequence of labeled irreducible tournaments (\href{https://oeis.org/A054946}{A054946} in the OEIS~\cite{OEIS2026}) is
 \[
  (\it_k) = 1,\, 0,\, 2,\, 24,\, 544,\, 22\,320,\, 1\,677\,488,\, 236\,522\,496,\, 64\,026\,088\,576,\, \ldots
 \]
 Hence, Corollary~\ref{proposition:graph} is consistent with the result of Wright~\cite{Wright1970apr}:
 \[
  \PP\big(\mbox{g is connected}\big) =
   1 -
    \binom{n}{1}\cdot2^{1-n} -
    2\cdot\binom{n}{3}\cdot2^{6-3n} -
    24\cdot\binom{n}{4}\cdot2^{10-4n} +
    O\big(n^5\cdot2^{-5n}\big).
 \]
 The fact that there are no irreducible tournaments of size $2$ is witnessed
  by the fact that the second nontrivial term (in $2^{-2n}$) is hidden.
\end{remark}

\begin{remark}
  Theorem~\ref{theorem:SET-asymptotics} can be applied to several other models
  of graphs.
  For example \cite[Statement 10.1.18]{Nurligareev2022} deals with multigraphs
  with distinguished edges (which is essentially
  Proposition~\ref{proposition:multigraph} with $d=2^k-1$).
  We can also apply it to estimate the probability that a directed graph is
  weakly connected.
  Let us recall that a \emph{digraph} of size $n$ is a directed graph defined
    on the set of vertices $[n]$, such that any ordered pair of distinct
    vertices $i\ne j$ is joined by at most one directed edge in each direction.
    An \emph{oriented graph} is a digraph with at most one directed edge
    between two distinct vertices.
    We have:

\begin{corollary}\label{proposition:oriented:graph}
 The asymptotic probability that a random labeled oriented graph $g$ with $n$~vertices is weakly connected satisfies
 \begin{equation}\label{formula: SET-asymptotics for oriented graphs}
   \PP\big(g\mbox{ is weakly connected}\big) \approx 1 - \sum\limits_{k\ge1} \it_k(2) \cdot \binom{n}{k} \cdot \frac{3^{k(k+1)/2}}{3^{kn}},
 \end{equation}
  where $\it_k(2)$ denotes the number of irreducible tournaments with ties of size $k$.
\end{corollary}
\begin{proof}
  Apply Proposition~\ref{proposition:multigraph} with $d=2$.
  The class of oriented graphs can be identified with the class of
  $2$-multigraphs, since in both cases, there are $3$ ways to link two distinct
  vertices.
\end{proof}

\begin{corollary}\label{proposition:digraph}
 The asymptotic probability that a random digraph $g$ with $n$~vertices is weakly connected satisfies
 \begin{equation}\label{formula: SET-asymptotics for digraphs}
   \PP\big(g\mbox{ is weakly connected}\big) \approx 1 - \sum\limits_{k\ge1} \it_k(3) \cdot \binom{n}{k} \cdot \frac{4^{k(k+1)/2}}{4^{kn}},
 \end{equation}
  where $\it_k(3)$ denotes the number of irreducible tournaments recording clean sheets of size $k$.
\end{corollary}
\begin{proof}
Apply Proposition~\ref{proposition:multigraph} with $d=3$.
\end{proof}

Note however that Theorem~\ref{theorem:SET-asymptotics} is not adapted to
  estimating the probability that a directed graph is strongly connected.
  Indeed, a digraph can be decomposed as a directed acyclic graph ($\DAG$) of
  its strongly connected components, not a $\SET$.
  The ideas presented here can however be adapted in such a more complex
  context~\cite{DovgalNurligareev2022}.
\end{remark}

\subsection{Surfaces}

\subsubsection{Square-tiled surfaces}\label{section:origami}

Square-tiled surfaces play a key role in the study of abelian and quadratic
differentials on Riemann surfaces.
They first appeared in disguise in the proof by Douady and Hubbard that among
quadratic differentials, the ones that can be decomposed into a disjoint union
of vertical cylinders (Jenkins--Strebel forms) is dense
\cite{DouadyHubbard1975}.
The dense subset they used is the set of quadratic differentials whose
(relative) periods have rational coordinates, which turn out to be square-tiled
surfaces. 
Beyond density, Eskin and Okounkov used the uniform distribution of such
surfaces to estimate the volume of the strata of
abelian~\cite{EskinOkounkov2001} and quadratic~\cite{EskinOkounkov2006}
differentials (the computation of explicit values was implemented by
Goujard~\cite{Goujard2016}).
Square-tiled surfaces also provide examples of differentials with exceptional
behavior, such as the ``eierlegende
Wollmilchsau''~\cite{HerrlichSchmithuesen2008} or the
``Ornithorynque''~\cite{ForniMatheusZorich2011}.
%

%
  A \emph{quadratic square-tiled surface} is obtained from finitely many
  squares by identifying pairs of sides of the squares by isometries in such
  way that horizontal sides are glued to horizontal sides and vertical sides to
  vertical.
%
%
  An \emph{abelian square tiled surface} (also known as \emph{origami}) is
  obtained by gluing finitely many squares such that left edges are glued to
  right edges and top edges are glued to bottom edges.
  Equivalently, an origami of size $n$ is a pair of permutations
  $(\sigma,\tau)$ on the ground set $[n]$, where $\sigma(i)=j$ if the right
  side of the square labeled $i$ is glued to the left side of the square
  labeled $j$, and $\tau(i)=k$ if the top side of the square labeled $i$ is
  glued to the bottom side of the side labeled $k$~\cite{Zmiaikou2011}.
  \begin{figure}[h!]
  \begin{center}
  \begin{tikzpicture}
    \foreach \x in {1,...,4}
    {
      \draw (\x,0) -- (\x+1,0) -- (\x+1,1) -- (\x,1) -- (\x,0);
      \draw (\x+0.5,0.5) node {$\x$};
      \draw (\x+3,1) -- (\x+4,1) -- (\x+4,2) -- (\x+3,2) -- (\x+3,1);
      \pgfmathtruncatemacro{\shifted}{\x+4};
      \draw (\x+3+0.5,1+0.5) node {$\shifted$};
    }
      \draw (1,0.5) node {\small $=$};
      \draw (5,0.5) node {\small $=$};
      \draw (4,1.5) node {\small $-$};
      \draw (8,1.5) node {\small $-$};
      \draw (1.5,1) node {\tiny $/$};
      \draw (2.5,1) node {\tiny $//$};
      \draw (3.5,1) node {\tiny $///$};
      \draw (7.5,1) node {\tiny $/$};
      \draw (6.5,1) node {\tiny $//$};
      \draw (5.5,1) node {\tiny $///$};
      \draw (1.5,0) node {\tiny $\backslash$};
      \draw (2.5,0) node {\tiny $\backslash\backslash$};
      \draw (3.5,0) node {\tiny $\backslash\backslash\backslash$};
      \draw (4.5,0) node {\tiny $\backslash\backslash\backslash\backslash$};
      \draw (4.5,2) node {\tiny $\backslash\backslash$};
      \draw (5.5,2) node {\tiny $\backslash$};
      \draw (6.5,2) node {\tiny $\backslash\backslash\backslash\backslash$};
      \draw (7.5,2) node {\tiny $\backslash\backslash\backslash$};
  \end{tikzpicture}
  \caption{the eierlegende Wollmilchsau~\cite{HerrlichSchmithuesen2008} is
    obtained from eight squares where two edges with the same decoration are
    identified by translation. It can be described as the pair of permutations
    $(\sigma,\tau)=((2,3,4,1,6,7,8,5),(8,7,6,5,2,1,4,3))$}
  \end{center}
  \end{figure}
  In particular, there are $((2n-1)!!)^2$ labeled quadratic square-tiled
  surface of size $n$, and $(n!)^2$ origamis of size $n$.
  We denote by $\O$ and $\CO$ the combinatorial classes of origamis and
  connected origamis, respectively.
%

Note that, in contrast with the historical references that focus on a fixed
stratum, we consider a model with unconstrained genus: every origami involving
$n$ labeled squares appears with the same probability (generically, the genus
grows in $\Theta(n)$, its distribution is studied in \cite{Shrestha2022}).

%
  On the derivative side, a permutation $\sigma$ on $[n]$ is
  \emph{indecomposable} if there is no $k<n$ such that
  $\sigma\big([k]\big)=[k]$~\cite{Comtet1974}.
  The class of indecomposable permutations is denoted by $\IP$.
  For an integer $d \geq 1$, a \emph{$d$-multiple linear order} of size $n$ is
  a $d$-tuple of linear orders of the set $[n]$.
  A $d$-multiple linear order $L = (<_1,\ldots,<_d)$ is \emph{reducible}, if
  there is a non-trivial partition $A\sqcup B$ of the ground set $[n]$ such
  that $a <_k b$ for every pair $(a,b) \in A \times B$ and any order $<_k$.
  Otherwise, we call $L$ \emph{irreducible}.
  We denote by $\L(d)$ and $\IL(d)$ the combinatorial classes of $d$-multiple
  linear orders and irreducible $d$-multiple linear order, respectively.
%

\begin{proposition}\label{proposition:origami}
  The asymptotic probability that an origami $o$ made up of $n$~unit squares is
  connected satisfies
  \begin{eqnarray}
    \PP\big(o\mbox{ is connected}\big)
    & \approx & 1 - \sum\limits_{k\ge1}\dfrac{\il_k(2)}{k!\cdot(n)_k}\\
    & \approx & 1 - \sum\limits_{k\ge1}\dfrac{\ip_k}{(n)_k},
  \end{eqnarray}
  where $\il_k(2)$ is the number of irreducible pairs of linear orders of size
  $k$,
  and $\ip_k$ is the number of indecomposable permutations of size $k$.
\end{proposition}

\begin{proof}
  Since an origami can be decomposed into the disjoint union of its connected
  components, we have $\O = \SET(\CO)$.
  Since each origami is determined by a pair of permutations, its counting
  sequence is $\o_n=(n!)^2$.
  Hence, $\O = \L(2)$, since the number of pairs of linear orders of size $n$
  is also $(n!)^2$.
  This change of perspective, from permutations to linear orders, allows to
  decompose with respect to the SEQ construction: we have $\L(2) =
  \SEQ(\IL(2))$.
  %
  
  %
  Let us show that the sequence $a_n = (n!)^2/n! = n!$ is gargantuan.
  Indeed, the first condition of Definition~\ref{def: gargantuan sequence} trivially holds.
  To check the second condition, note that $a_ka_{n-k} \geq a_{k+1}a_{n-k-1}$ as long as $k \leq (n-1)/2$.
  Note also, that $na_{n-r-1} = O(a_{n-r})$.
  Hence, replacing $a_ka_{n-k}$ by $a_{r+1}a_{n-r-1}$ for all $k$ such that $r<k<n-r$, we have
  \[
    \sum\limits_{k=r}^{n-r}a_ka_{n-k}
     \leq
    2a_ra_{n-r} + (n-2r+1)a_{r+1}a_{n-r-1}
     =
    O(a_{n-r}).
  \]
  Thus, Theorem~\ref{theorem:SET-asymptotics} is applicable and gives the following asymptotics:
 \[
  \PP\big(o\mbox{ is }\SET\mbox{-irreducible}\big)
   \approx
  1 - \sum\limits_{k\ge1} \il_k(2) \cdot \binom{n}{k} \cdot \frac{\o_{n-k}(d)}{\o_n(d)}
  =
  1 - \sum\limits_{k\ge1} \dfrac{\il_k(2)}{k! \cdot (n)_k}.
 \]
\newline
The second asymptotics follows from
(see Remark~\ref{remark:lift})
\begin{equation}\label{eq:ip=il(2)/p}
  \ip_n = \dfrac{\il_n(2)}{n!}.
\end{equation}
\end{proof}

\begin{remark}\label{remark:lift}
  The nature of the elements of the ground set $[n]$ plays a role in the
  definition of indecomposable permutations, as those labels need to be
  compared with each other.
  The ground set both serves as a label set and as a way to compare elements:
  the notion of indecomposable permutation is not stable by relabeling.
  In particular, it is not possible to construct the product or the sequence of
  indecomposable permutations, since this operation involves relabelling (for
  example, how to define the sequence of the three irreducible permutations
  $321, 21, 3142$?).
  The class $\IP$ can therefore not be considered as a combinatorial class as
  its definition is not functorial (we can not transport its structure to any
  label set).
  We can see the class $\IL(2)$ as a \emph{lift} of the class $\IP$ that
  consists in decoupling the role of the ground set $[n]$ as a label set on the
  one hand, and as a linear order on the other.
  The first order $<_1$ can be interpreted as the linear order of the ground
  set $1<2<\dots<n$, and the second as the linear order of the image by the
  permutation $\sigma(1)<\sigma(2)<\dots<\sigma(n)$.
  The $n!$ available relabelings explain the relation~\eqref{eq:ip=il(2)/p}.
\end{remark}

\begin{remark}
  Proposition~\ref{proposition:origami} can be seen as a topological version of
  results of Comtet, Dixon and Cori.
  Indeed, the connectedness of an origami $o$ of size $n$ is equivalent to the
  fact that the two permutations in the pair $(\sigma,\tau)$ determining $o$ generate a
  transitive subgroup of $S_n$.
  In 2005, Dixon~\cite{Dixon2005} showed that the asymptotic probability $t_n$
  that two permutations of size $n$ form a transitive subgroup satisfies \[ t_n
  \sim 1 - \dfrac{1}{(n)_1} - \dfrac{1}{(n)_2} - \dfrac{3}{(n)_3} -
  \dfrac{13}{(n)_4} - \dfrac{71}{(n)_5} - \dfrac{461}{(n)_6} - \ldots, \] where
  the numerators are coefficients of the multiplicative inverse of the formal
  power series comprised by factorials.
  The latter, as it was showed by Comtet~\cite{Comtet1972} in 1972, count
  indecomposable permutations.
  Alternatively, this fact was established by Cori~\cite{Cori2009} in~2009.
\end{remark}

\subsubsection{Combinatorial maps}\label{section:combinatorial-map}

There are several models of random discrete oriented surfaces.
The simplest model consists of oriented triangles glued together along their
edges.
It was introduced in 2004 by Brooks and Makover~\cite{BrooksMakover2004} in
order to study the ``typical'' Riemann surfaces with high genus, and
independently by
Pippenger and Schleich~\cite{PippengerSchleich2005} who were motivated by the
needs of quantum gravity. Pippenger and Schleich showed that a random gluing of
$2n$ triangles form a connected surface with probability
\begin{equation}\label{formula: Pippenger and Schleich asymptotics}
  1 - \frac{5}{36n} + O\left(\frac{1}{n^2}\right)
\end{equation}
and studied their topological characteristics.
Gamburd~\cite{Gamburd2006} generalized this model by replacing triangles by
$p$-gons for a fixed integer $p\ge3$.
Then, Chmutov and Pittel~\cite{ChmutovPittel2016} extended the model to
polygons whose set of allowed perimeters is a given nonempty subset $J$ of
$\NN_{\geq 3}$, and showed that the probability to get a connected surface
always satisfies
$$1 - O\left(\frac{1}{n}\right).$$
Finally, in 2019, Budzinski, Curien and Petri~\cite{BudzinskiCurienPetri2019}
considered a model without any constraint on the perimeters of the polygons
(1-gons and 2-gons are permitted), and found that the probability to obtain a
connected surface is
\begin{equation}\label{formula: Budzinski, Curien and Petri asymptotics}
  1 - \frac{1}{n} + O\left(\frac{1}{n^2}\right).
\end{equation}
The main goal of the above authors was to study topological characteristics of
the proposed models, such as their Euler characteristic, genus, diameter, etc.
In all the mentioned models, glued polygons are assumed to be oriented and the
identified sides are oriented opposite-wise.
Let us focus on this last model studied by Budzinski, Curien and Petri.
%

%
  A \emph{perfect matching} is a permutation whose orbits have length $2$.
  We denote by $\IM$ the class of indecomposable perfect matchings.
  A \emph{combinatorial map} of size $n$ is obtained from a collection of
  polygons whose set of sides is labeled by integers from $1$ to $n$, by
  identifying in pairs the sides of these polygons.
  Equivalently, a combinatorial map of size $n$ is a pair of permutations
  $(\sigma, \alpha)$ on $[n]$ such that $\alpha$ is a perfect matching.
  The orbits of $\sigma$ correspond to the enumeration of the sides of the
  polygons in the cyclic order. The orbits of $\alpha$ correspond to the
  identification of the edges.
  We denote by $\CM$ and $\CCM$ the classes of combinatorial maps and connected
  combinatorial maps, respectively.
  %
  %

There are $n!(n-1)!!$ combinatorial maps of size $n$ when $n$ is even.
There is no combinatorial map of odd size because the sides of the polygons they are made of are identified in pairs.
The elements of the ground set $[n]$ are sometimes called ``darts'' or
``half-edges''.

\begin{definition}\label{def:linear-matching}
  A \emph{linear matching} is a pair of linear orders $(<_1,<_2)\in\L(2)$ such
  that the transform \[ (<_1,<_2) \mapsto (<_2,<_1) \] coincides with some
  relabeling $(1,\ldots,n)\mapsto(1,\ldots,n)$ that is an involution without any fixed point.
  We denote by $\LM(2)$ and $\ILM(2)$ the classes of linear matchings and
  irreducible linear matchings, respectively.
\end{definition}

\begin{proposition}\label{proposition:SET-asymptotics for combinatorial map surfaces}
 The asymptotic probability that a random combinatorial map $m\in\CM$ of size $2n$ is connected satisfies
   \begin{eqnarray}\label{formula: SET-asymptotics for combinatorial map surfaces}
    \PP\big(m\mbox{ is connected}\big)
    & \approx & 1 - \sum\limits_{k\ge1} \dfrac{\ilm_{2k}}{(2k)!}\cdot\dfrac{\big(2(n-k)-1\big)!!}{(2n-1)!!}\\
    & \approx & 1 - \sum\limits_{k\ge1} \im_{2k}\cdot\dfrac{\big(2(n-k)-1\big)!!}{(2n-1)!!},
  \end{eqnarray}
 where $\ilm_{k}$ is the number of irreducible linear matchings,
 and $\im_{k}$ is the number of indecomposable perfect matchings of size $k$.
\end{proposition}

\begin{proof}
 We have $\CM = \LM(2)$, since the 2-periodic counting sequence of both classes
  is $n!\cdot(n-1)!!\cdot\II_{\{n\mbox{\footnotesize{ }is even}\}}$.
  As in the proof of Proposition~\ref{proposition:origami}, we have the double
  $\SET/\SEQ$ decomposition $\CM = \SET(\CCM)$ and $\LM(2) = \SEQ(\ILM(2))$.
  \newline
 Let us show that the sequence
 \[
  a_n = \dfrac{(2n)! \cdot (2n-1)!!}{(2n)!} = (2n-1)!!
 \]
 is gargantuan.
 To this end, we apply Lemma~\ref{lemma: sufficient conditions for gargantuan sequence}.
 The first condition holds, since
 \[
  na_{n-1} = n(2n-3)!! =
  O\big((2n-1)!!\big) = O(a_n).
 \]
 Regarding the second condition, for $k \le (n-1)/2$ we have
 \[
  \dfrac{x_{k+1}}{x_{k}} =
  \dfrac{(2k+1)!!(2n-2k-3)!!}{(2k-1)!!(2n-2k-1)!!} =
  \dfrac{2k+1}{2n-2k-1} \le 1.
 \]
 Therefore, $(x_k)$ is decreasing for $k < n/2$, and the sequence $(a_n)$ is gargantuan.
\newline
\newline
 Hence, we can apply Proposition~\ref{theorem:SET-asymptotics, d-gargantuan} to obtain
 \[
  \PP\big(s\mbox{ is connected}\big)
  \approx
 1 - \sum\limits_{k\ge1}
  \dfrac{\ilm_{2k}}{(2k)!}\cdot\dfrac{\big(2(n-k)-1\big)!!}{(2n-1)!!}.
 \]
 The second asymptotics follows from $\ilm_n(2) = n!\cdot\im_n$, as explained
  in Remark~\ref{remark:lift} for~\eqref{eq:ip=il(2)/p}.
\end{proof}

\subsection{Higher dimensions}

\subsubsection{Graph Encoded Manifolds}\label{section:GEM}

GEMs, for Graph Encoded Manifolds, were introduced in the framework of
``crystallization theory'' as a way to encode compact
PL-manifolds~\cite{Pezzana1974,Pezzana1975,FerriGagliardiGrasselli1986}.
Those objects recently attracted attention of theoretical physicists, where
GEMs encode colored tensor models, seen as quantum gravity
theories~\cite{Gurau2011,BonzomGurauRielloRivasseau2011,GurauRyan2012,Witten2019}.

%
  Let us fix a dimension $\DDIM\geq2$ and a size $n\geq 1$.
  A simplex of dimension $\DDIM$ has its vertices colored from $1$ to
  $\DDIM+1$.
  A \emph{GEM of dimension $\DDIM$} of size $n$ is a simplicial complex
  obtained from $n$ copies of the simplex glued according to the following
  rule: for each $1\leq k\leq\DDIM+1$, consider a perfect matching $\alpha_k$
  of $[n]$ and glue the hyperface opposite to the $k$th vertex of the $i$th
  simplex to the hyperface that opposite to the $k$th vertex of the
  $\alpha_k(i)$th simplex.
  We denote by $\GEM(\DDIM)$ the class of GEMs of dimension $\DDIM$.
  Equivalently, a $GEM$ of dimension $\DDIM$ and size $n$ is a graph on $[n]$
  vertices, where each vertex has a degree $(\DDIM+1)$ and each edge is colored
  into one of $\DDIM+1$ colors such that every vertex is incident to exactly
  one edge of each color.
%

\begin{figure}[h!]
\begin{center}
\begin{tikzpicture}[scale=0.6]
  \draw[fill=lightgray!50] (-3,5) -- (-1,0) -- (-5,-2) -- cycle ;
  \draw (-3,5) -- (-1,0) -- (-5,-2) -- (-7,2) -- (-3,5) -- (-5,-2) ;
  \draw[dashed] (-7,2) -- (-1,0) ;
  \draw[right] (-1,0) node {$0$};
  \draw[above] (-3,5) node {$1$};
  \draw[left] (-7,2) node {$2$};
  \draw[below] (-5,-2) node {$3$};
  \draw[fill=lightgray!50] (+3,5) -- (+1,0) -- (+5,-2) -- cycle ;
  \draw (+3,5) -- (+1,0) -- (+5,-2) -- (+7,2) -- (+3,5) -- (+5,-2) ;
  \draw[dashed] (+7,2) -- (+1,0) ;
  \draw[left] (+1,0) node {$0$};
  \draw[above] (+3,5) node {$1$};
  \draw[right] (+7,2) node {$2$};
  \draw[below] (+5,-2) node {$3$};
  \draw[<->, thick] (-3,2) to [bend right=20] (3,2) ;
  \draw (0,1.7) node {$\alpha_2$};
\end{tikzpicture}
  \caption{gluing of two simplices along the faces opposite to the $2$nd vertex}
\end{center}
\end{figure}

A GEM has an even number of simplices.
Every compact PL-manifold can be represented as a GEM. 
%
%
A GEM is orientable if, and only if, the associated graph is
bipartite~\cite{CavicchioliGrasselli1980}: there exists a partition of the
simplices into two sets $A$ and $B$ of size $n/2$ such that the simplices in
$A$ the are glued to simplices of $B$.
We denote $\OGEM(\DDIM)$ and $\COGEM(\DDIM)$ the class of orientable GEMs and
connected orientable GEMs of dimension $\DDIM$, respectively.
In the colored tensor model, orientable GEMS of dimension $\DDIM$ are called
\emph{closed (D+1)-colored graphs}.
There are $((2n-1)!!)^{\DDIM+1}$ GEMs and
$\binom{2n}{n}(n!)^{\DDIM+1}=(2n!)(n!)^{\DDIM-1}$ orientable GEMs of dimension $\DDIM$
and size $2n$.

In 2019, within the framework of the colored tensor model,
Carrance~\cite{Carrance2019} showed that a random orientable GEM of dimension
$D$ and size $n$ is connected with probability
\begin{equation}\label{formula:Carrance}
  1 - \frac{1}{n^{\DDIM-1}} + O\left(\frac{1}{n^{2(\DDIM-1)}}\right).
\end{equation}

The whole asymptotic expansion is provided by the following proposition.

\begin{proposition}\label{proposition:OGEM}
  Let $\DDIM\ge2$.
  The asymptotic probability that a random orientable GEM $g$ of dimension $D$
  of size $2n$ is connected satisfies
  \begin{equation}\label{formula:OGEM}
    \PP(s\mbox{ is connected}) \approx 1 -
    \sum\limits_{k\ge1}\dfrac{\il_{k}(\DDIM)}{k! \cdot
    \big((n)_k\big)^{\DDIM-1}}.
  \end{equation}
\end{proposition}

\begin{proof}
 The counting sequence of the class $\OGEM(\DDIM)$ is 2-periodic with
 \[
  \ogem_{n}(\DDIM+1) =
  \binom{2n}{n} (n!)^{\DDIM+1} =
  (2n)! (n!)^{\DDIM-1}.
 \]
 Hence, the exponential generating function of this class coincides with the one of the class $\L(\DDIM)$ of $\DDIM$-multiple linear orders taken at $z^2$:
 \[
  \sum\limits_{n=0}^{\infty} (n!)^{\DDIM-1} z^{2n}.
 \]
 We have the following double $\SET/\SEQ$ decompositions:
 $\OGEM(\DDIM) = \SET(\COGEM(\DDIM))$ and
 $\L(\DDIM) = \SEQ(\IL(\DDIM))$.
 We have seen in the proof of Proposition~\ref{proposition:origami} that the sequence $(n!)$ is gargantuan.
 Hence, according to Lemma~\ref{lemma: a_nb_n is gargantuan}, the sequence
 \[
  \dfrac{\ogem_{2n}(\DDIM)}{(2n)!} = (n!)^{\DDIM-1}
 \]
 is gargantuan too.
\newline
\newline
 This allows us to apply Proposition~\ref{theorem:SET-asymptotics, d-gargantuan}.
 Note that we need to formally rewrite the exponential generating function of the class $\IL(\DDIM)$ taken at $z^2$ as if it corresponds to a combinatorial class with 2-periodic counting sequence:
 \[
  \sum\limits_{n=1}^{\infty} \dfrac{(2n)! \cdot \il_k(\DDIM)}{n!} \cdot \dfrac{z^{2n}}{(2n)!}.
 \]
 Doing so, we get the asymptotics in the form
\begin{eqnarray*}
  \PP\big(g\mbox{ is }\SET\mbox{-irreducible}\big) & \approx &
  1 - \sum\limits_{k\ge1} \dfrac{(2k)! \cdot \il_k(\DDIM)}{k!} \cdot \binom{2n}{2k} \cdot \frac{\ogem_{n-k}(\DDIM)}{\ogem_n(\DDIM)} \\
  & \approx &
  1 - \sum\limits_{k\ge1} \dfrac{\il_k(\DDIM)}{k!} \cdot \frac{1}{\big((n)_k\big)^{\DDIM-1}}. \\
\end{eqnarray*}
\end{proof}

\begin{remark}
 We can extend the notion of indecomposable permutation in higher dimension.
 Let us define a $d$-\emph{multipermutation} of size $n$ as a $d$-tuple of permutations
 \[
  (\sigma_1,\ldots,\sigma_d) \in S_n^d.
 \]
 A $d$-multipermutation is said \emph{indecomposable}, if there is no $k<n$ such that $\sigma_i([k])=[k]$ for all $i\in[d]$.
 The same way as for relation~\eqref{eq:ip=il(2)/p}, the counting sequence of indecomposable $d$-multipermutations satisfies
 \begin{equation} \label{eq:imp(d)=il(d+1)/p}
  \imp_n(d) = \dfrac{\il(d+1)}{n!}
 \end{equation}
 This fact allows us to reformulate Proposition~\ref{proposition:OGEM} in the following way.

\begin{corollary}\label{proposition:OGEM2}
  Let $\DDIM\ge2$.
  The asymptotic probability that a random orientable GEM $g$ of dimension $D$
  of size $2n$ is connected satisfies
  \begin{equation}\label{formula:OGEM2}
    \PP(s\mbox{ is connected}) \approx 1 -
    \sum\limits_{k\ge1}\dfrac{\imp_{k}(\DDIM-1)}{\big((n)_k\big)^{\DDIM-1}}.
  \end{equation}
  where $\imp_{k}(\DDIM-1)$ denotes the number of indecomposable
  $(\DDIM-1)$-multipermutations of size $k$.
\end{corollary}

\end{remark}

\subsubsection{Constellations}


  %
  Given a positive integer $\DIM$, a $\DIM$-\emph{constellation} of size $n$ is
  a $\DIM$-mulipermutation $(\sigma_1,\ldots,\sigma_{\DIM})\in S_n^{\DIM}$ such that:
  \begin{itemize}
    \item the group $\langle\sigma_1,\ldots,\sigma_{\DIM}\rangle$ acts
      transitively on the set $[n]$,
    \item the product of $\sigma_k$ is the identity permutation, i.e.
      $\sigma_1\ldots\sigma_{\DIM}=\mathrm{id}$.
  \end{itemize}
  We denote the class of $\DIM$-constellations by $\CN(\DIM)$.
 %

While constellation have been introduced to study unramiﬁed covering of the
punctured sphere~\cite{LandoZvonkin2004}, they can be interpreted as a higher
dimensional generalization of origamis studied in
Section~\ref{section:origami}.
Indeed, the second condition is equivalent to
$\sigma_{\DIM}=(\sigma_1\ldots\sigma_{\DIM-1})^{-1}$ so that the first $\DIM-1$
permutations can be chosen independently, and the last one is determined by
those.
The permutations $\sigma_1,\ldots,\sigma_{\DIM-1}$ of $[n]$ describe how $n$
$(\DIM-1)$-hypercubes are glued in each direction to form a generalized
origami.
In this interpretation, the first condition is equivalent to the connectedness
of the generalized origami.

\begin{proposition}\label{proposition:constellation}
  Let $\DIM\ge3$.
  The number $c_n(\DIM)$ of $d$-constellations of size $n$ satisfies:
  \begin{eqnarray}
  c_n(d)
    & \approx & (n!)^{\DIM-1}\left(1-\sum\limits_{k\ge1}\dfrac{\imp_k(\DIM-2)}{\big((n)_k\big)^{\DIM-2}}\right)\label{formula:constellation:1}\\
    & \approx & (n!)^{\DIM-1}\left(1-\sum\limits_{k\ge1}\dfrac{\il_k(\DIM-1)}{k!\cdot\big((n)_k\big)^{\DIM-2}}\right)\label{formula:constellation:2}.
  \end{eqnarray}
\end{proposition}


\begin{proof}
  The estimation of the number of constellations follows from the computation
  of the probability that a uniformly chosen $(\DIM-1)$-multipermutation
  determines $\DIM$-constellation (after adding the last permutation).
  This result is a generalisation of Proposition~\ref{proposition:origami}.
  If $\S(\DIM-1)$ denotes the class of $(\DIM-1)$-multipermutations, we have
  $\S(\DIM-1)=\SET\big(\CN(\DIM)\big)$.
  The corresponding counting sequence is~$(n!)^{\DIM-1}$, which coincides with
  the one of the class $\L(\DIM-1)=\SEQ(\IL(\DIM-1))$.
  %
  %
  %
  As we have seen in the proof of Proposition~\ref{proposition:OGEM}, the
  sequence~$(n!)^{\DIM-2}$ is gargantuan for $\DIM>2$.
  Therefore, we can apply Theorem~\ref{theorem:SET-asymptotics} that, together
  with relation~\eqref{eq:imp(d)=il(d+1)/p} and multiplication of both sides
  by~$(n!)^{\DIM-1}$, gives the asymptotic
  expansions~(\ref{formula:constellation:1}) and
  (\ref{formula:constellation:2}). 
\end{proof}

\section{Conclusion}

We have seen how to provide the whole asymptotic expansion of the probability
that a combinatorial class admitting a double SET/SEQ decomposition is
connected.
The relationship between those two decompositions corresponds, symbolically, to
the relation
$$\frac{\partial\log(z)}{\partial z} = \frac{1}{z}.$$
\newline

It is tempting to iterate this idea to get the asymptotic probability that a
combinatorial structure is SEQ-irreducible from the equation
$$\dfrac{\partial}{\partial z}\left(\dfrac{1}{z}\right) = - \frac{1}{z^2}.$$
The derivative sequence should therefore be related to structures admitting a
decomposition into two SEQ-irreducible components.
This will be the subject of the forthcoming paper~\cite{MonteilNurligareevSEQ}.
\newline

In the applications, we relied on the existence of a double $\SET/\SEQ$
decomposition for some combinatorial classes: graphs \emph{vs} tournaments,
permutations \emph{vs} linear orders, and combinatorial maps \emph{vs} linear
matchings.
However, some models ($p$-angulated surfaces, quadratic square-tiled surface,
not necessarily orientable GEMs) have been described without an explicit SEQ
decomposition.
First, note that it is still possible to apply (the proof of)
Theorem~\ref{theorem:SET-asymptotics}, by replacing the combinatorial sequence
$\d_k$ with the $k$th coefficient of the formal power series $D(z):=1-1/A(z)$.
In particular, we have:

\begin{proposition}\label{proposition:noSEQ}
%
 $$
  \PP(\mbox{a triangulated surface is connected}) \approx
  1 -
  \frac{5}{36n} -
  \frac{695}{2592n^2} - 
  \frac{216305}{279936n^3} - 
  \ldots
 $$
%
%
  \newline
 $$
  \PP(\mbox{a quadrangulated surface is connected}) \approx
  1 -
  \frac{3}{16n} -
  \frac{183}{512n^2} -
  \frac{8313}{8192n^3} -
  \ldots
 $$
%
%
  \newline
 $$
  \PP(\mbox{a quadratic square-tiled surface is connected}) \approx
  1 -
  \frac{1}{4n} -
  \frac{15}{32n^2} -
  \frac{167}{128n^3} -
  \frac{11845}{2048n^4} -
  \ldots
 $$
%
%
  \newline
%
 $$
  \PP(\mbox{a GEM of dimension 3 is connected}) \approx 
  1 -
 \frac{1}{8n^2} -
 \frac{3}{16n^3} -
 \frac{49}{128n^4} -
 \frac{145}{128n^5} -
  \ldots
 $$
\end{proposition}
\begin{proof}
  Apply the proof of Theorem~\ref{theorem:SET-asymptotics},
  where $\d_k/k!$ is the $k$th coefficient the of
  $$
   1 - \left(
    \sum\limits_{n\geq0}(6n-1)!!\dfrac{z^{2n}}{(2n)!}
   \right)^{-1},
  $$
  $$
   1 - \left(
    \sum\limits_{n\geq0}(4n-1)!!\dfrac{z^n}{n!}
   \right)^{-1}
  $$
  $$
   1 - \left(
    \sum\limits_{n\geq0}\big((2n-1)!!\big)^2 \dfrac{z^n}{n!}
   \right)^{-1},
  $$
  and
  $$
   1 - \left(
    \sum\limits_{n\geq0}\big((2n-1)!!\big)^4 \dfrac{z^{2n}}{(2n)!}
   \right)^{-1},
  $$
  respectively.
\end{proof}

However, a combinatorial interpretation for the coefficients is still missing.
In the second forthcoming paper~\cite{MonteilNurligareevANTISEQ}, we will construct an
``$\mbox{anti-}\SEQ$'' operator, which provides a combinatorial interpretation
of such classes.

\bibliographystyle{abbrv}
\bibliography{bibliography}

\end{document}